\documentclass[12 pt]{article}
\usepackage{fullpage}

\usepackage{lipsum}
\usepackage{amsfonts}
\usepackage{graphicx}
\usepackage{epstopdf}
\usepackage{standalone}
\usepackage{algcompatible}

\ifpdf
\DeclareGraphicsExtensions{.eps,.pdf,.png,.jpg}
\else
\DeclareGraphicsExtensions{.eps}
\fi

\usepackage{amsopn}

\usepackage{epsfig} %
\usepackage{amsmath} %
\usepackage{amssymb}  %
\usepackage{bm}
\DeclareMathAlphabet{\mathcal}{OMS}{cmsy}{m}{n}
\usepackage{algorithm}
\usepackage{color}
\usepackage{pgfplots}
\usepackage{siunitx}
\usepackage{tikz}
\usepackage{tkz-euclide}
\usepackage{chngcntr}
\usepackage{verbatim}
\usepackage{enumerate}

\definecolor{ao(english)}{rgb}{0.0, 0.5, 0.0}
\usepackage{cleveref}
\usepackage{aliascnt}
\usepackage{subcaption}
\usepackage{tikz}
\usepackage{pgfplots}
\usetikzlibrary{arrows,shapes,trees,calc,positioning,patterns,decorations.pathmorphing,decorations.markings}
\usetikzlibrary{matrix}
\usepgfplotslibrary{groupplots}
\pgfplotsset{compat=newest}
\usepackage{amsthm}
\usepgfplotslibrary{patchplots}

\crefname{figure}{Fig.}{Fig.}

\newtheorem{thm}{Theorem}
\crefname{thm}{Theorem}{Theorems}

\newtheorem{prop}{Proposition}
\crefname{prop}{Proposition}{Propositions}
\newtheorem{lem}{Lemma}
\crefname{lem}{Lemma}{Lemmas}
\newtheorem{cor}{Corollary}
\crefname{cor}{Corollary}{Corollaries}
\theoremstyle{remark}

\crefname{rem}{Remark}{Remarks}

\theoremstyle{definition}

\crefname{example}{Example}{Examples}

\crefname{ass}{Assumption}{Assumption}
\usepackage{dsfont}
\let\mathbb=\mathds

\theoremstyle{definition}

\crefname{defn}{Definition}{Definitions}

\crefname{prob}{Problem}{Problems}
\crefname{algorithm}{Algorithm}{Algorithms}

\newcommand{\Rn}{\mathbb{R}^{n \times n}_{\geq 0}}

\newcommand{\Rmn}{\mathbb{R}^{n \times m}}
\newcommand{\Rnv}{\mathbb{R}^{n}_{\geq 0}}

\newcommand{\Rnn}{\mathbb{R}^{n \times n}}

\newcommand{\Rtva}{\mathbb{R}^{2 \times 2}_{\geq 0}}

\newcommand{\Rtre}{\mathbb{R}^{3 \times 3}_{\geq 0}}
\newcommand{\Rtrev}{\mathbb{R}^{3}_{\geq 0}}
\newcommand{\Rfyra}{\mathbb{R}^{4 \times 4}_{\geq 0}}

\newcommand{\cone}{\textnormal{cone}}
\newcommand{\conv}{\textnormal{conv}}
\newcommand{\rk}{\textnormal{rank}}

\newcommand{\diag}{\textnormal{diag}}

\newcommand{\blkdiag}{\textnormal{blkdiag}}

\newcommand{\inter}{\textnormal{int}}

\newcommand{\transp}{\mathsf{T}}

\colorlet{FigColor1}{blue}
\colorlet{FigColor2}{red}
\colorlet{FigColor3}{ao(english)}
\colorlet{FigColor4}{orange}
\pgfplotsset{every axis plot/.append style={line width=1.5pt}}

\crefformat{equation}{\textup{#2(#1)#3}}
\crefrangeformat{equation}{\textup{#3(#1)#4--#5(#2)#6}}
\crefmultiformat{equation}{\textup{#2(#1)#3}}{ and \textup{#2(#1)#3}}
{, \textup{#2(#1)#3}}{, and \textup{#2(#1)#3}}
\crefrangemultiformat{equation}{\textup{#3(#1)#4--#5(#2)#6}}%
{ and \textup{#3(#1)#4--#5(#2)#6}}{, \textup{#3(#1)#4--#5(#2)#6}}{, and \textup{#3(#1)#4--#5(#2)#6}}

\Crefformat{equation}{#2Equation~\textup{(#1)}#3}
\Crefrangeformat{equation}{Equations~\textup{#3(#1)#4--#5(#2)#6}}
\Crefmultiformat{equation}{Equations~\textup{#2(#1)#3}}{ and \textup{#2(#1)#3}}
{, \textup{#2(#1)#3}}{, and \textup{#2(#1)#3}}
\Crefrangemultiformat{equation}{Equations~\textup{#3(#1)#4--#5(#2)#6}}%
{ and \textup{#3(#1)#4--#5(#2)#6}}{, \textup{#3(#1)#4--#5(#2)#6}}{, and \textup{#3(#1)#4--#5(#2)#6}}

\crefdefaultlabelformat{#2\textup{#1}#3}

\title{On the Similarity to Nonnegative and Metzler Hessenberg Forms%
	\thanks{Special Matrices, accepted for publication.
		
		This work was supported by the ELLIIT Excellence Center and by the Swedish Research Council through the LCCC Linnaeus Center. It was also supported by the European Research Council under the ERC Advanced Grant Agreements Switchlet n.670645 and ScalableControl n.834142 as well as by DGAPA-UNAM under the grant PAPIIT RA105518 and by SSF under the grant RIT15-0091 SoPhy."}}

\author{Christian Grussler \thanks{Department of Electrical Engineering and Computer Sciences, UC Berkeley, Berkeley, CA 
		({christian.grussler@berkeley.edu})}
	\and Anders Rantzer \thanks{Department of Automatic Control, Lund University, Lund, Sweden 
		({rantzer@control.lth.se}.)}}

\begin{document}

\maketitle

\begin{abstract}
We address the issue of establishing standard forms for nonnegative and Metzler matrices by considering their similarity to nonnegative and Metzler Hessenberg matrices. It is shown that for dimensions $n \geq 3$, there always exists a subset of nonnegative matrices that are not similar to a nonnegative Hessenberg form, which in case of $n=3$ also provides a complete characterization of all such matrices. For Metzler matrices, we further establish that they are similar to Metzler Hessenberg matrices if $n \leq 4$. In particular, this provides the first standard form for controllable third order continuous-time positive systems via a positive controller-Hessenberg form. Finally, we present an example which illustrates why this result is not easily transferred to discrete-time positive systems. While many of our supplementary results are proven in general, it remains an open question if Metzler matrices of dimensions $n \geq 5$ remain similar to Metzler Hessenberg matrices. 
\end{abstract}

\section*{Introduction}
Similarity transformation of matrices into structured standard forms is among the most fundamental tools in linear algebra. One of the most common real-valued forms is the \emph{upper Hessenberg form}, i.e., matrices which are zero except for possible non-zero elements above and on the first lower diagonal \cite{datta2004numerical}. Non-uniqueness of this form, however, raises the question whether it is possible to incorporate and preserve additional properties. In this work, we consider the following problem: 
\begin{center}
	\emph{When is a nonnegative matrix similar to a nonnegative Hessenberg matrix?}
\end{center}
We also ask this question when nonnegativity is replaced by Metzler, i.e., matrices that are nonnegative apart from their main diagonals. Besides a necessary condition for trace-zero nonnegative matrices in $\mathds{R}^{5 \times 5}_{\geq 0}$ \cite{reams1996inequality}, little appears to be known about this problem. 

Similar to standard numerical approaches that transform a nonnegative matrix $A \in \mathds{R}^{n \times n}_{\geq 0}$ into upper Hessenberg form \cite{datta2004numerical}, in this work, we seek a nonnegative similarity transformation $T_{2} \in  \mathds{R}^{n-1 \times n-1}_{\geq 0}$ such that the block-diagonal matrix $T = \begin{pmatrix}
	1 & 0 \\
	0 & T_2
\end{pmatrix}$ yields a nonnegative/Metzler
\begin{equation*}
	T^{-1}AT = T^{-1}\begin{pmatrix}
		a_{11} & c_{2} \\
		b_{2} & A_{2}
	\end{pmatrix}T = \begin{pmatrix}
		a_{11} & c_2 T_2\\
		T_2^{-1}b & T_2^{-1}A_2T_2
	\end{pmatrix}
\end{equation*}
with $T_2^{-1}b_2 = e_1$ being the first canonical unit vector in $\mathds{R}^{n-1}$ and $T_2^{-1}A_2T_2$ an upper Hessenberg matrix. As $T_2$ can be found recursively by replacing $A$ with $A_2$, our investigations will start with the case of $n=2$ to characterize the case $n=3$, which then leads to $n=4$. Our main complete characterizations are:
\begin{enumerate}[I.]
	\item $A \in \Rtre$ is similar to a nonnegative Hessenberg matrix if and only if $A$ does not have a negative eigenvalue of geometric multiplicity $2$.
	\item Every Metzler $A \in \mathds{R}^{n \times n}$ with $n \leq 4$ is similar to a Metzler Hessenberg matrix. \label{item:Metlzer_4D}
\end{enumerate}
Further, we note that all $A \in \mathds{R}^{n \times n}_{\geq 0}$ with a negative eigenvalue of geometric multiplicity $n-1$ cannot be similar to a nonnegative Hessenberg matrix.  

In the language of system and control theory, our approach is equivalent to finding a state-space transformation $T_{2} \in  \mathds{R}^{n-1 \times n-1}$ to a so-called \emph{discrete-time (DT) positive system} \cite{farina2011positive} $(A_2,b_2,c_2)$: a state-space system
\begin{equation*}
	\begin{aligned}
		x(t+1) &= A_2 x(t) + b_2 u(t),\\
		y(t) &= c_2x(t),
	\end{aligned}
\end{equation*}
with nonnegative system matrices  $A_2,b_2,c_2$; such that $(T_2^{-1}AT_2,T_2^{-1}b,cT_2)$ is a positive system in \emph{controller-Hessenberg form} \cite{datta2004numerical}, i.e., $T_2^{-1}AT$ is a nonnegative upper Hessenberg matrix, $T_2^{-1}b = e_1$ and $c_2T_2$ is nonnegative. In other words, our characterization also attempts to answer the question when a positive systems can be realized in positive controller standard form. So far, this question has only been implicitly considered for minimal second-order systems \cite{farina2011positive}. For \emph{continuous-time (CT) positive systems}, which are defined by replacing $A_2$ with a Metzler matrix and $x(t+1)$ with $\dot{x}(t)$, our results rely on a new characterization of all CT positive systems with state-space dimension $3$ that can be transformed into a positive system with controller-Hessenberg form. As this also remains true for DT positive systems, where $A_2$ is sufficiently diagonally dominant, we can further conclude that \Cref{item:Metlzer_4D} also applies for sufficiently diagonally dominant $A \in \mathds{R}^{4 \times 4}$. Unfortunately, as demonstrated by an illustrative example, a general characterization in DT is less trivial. 

While many of our techniques and tools can be extend to structured cases of larger dimensions, it remains an open question whether a Metzler Hessenberg form also exists for $n \geq 5$. The main reason for this analysis limitation is our usage of the fact that polyhedral cones, which are embedded in a 2-dimensional subspace, can be generated by at most two extreme rays. As this is not necessarily true for cones that are embedded in a higher dimensional subspace (see nonnegative rank \cite{cohen1993nonnegative}), our analysis does not simply extend. 

Finally, note that our results can also be seen towards standard forms for the nonnegative/Metzler inverse eigenvalue problem (see, e.g., \cite{loewy2017necessary,cronin2018diagonalizable,egleston2004nonnegative}), a problem which shares the issue that complete solutions for $n> 4$ are missing. 
\section{Preliminaries}
For real valued matrices $X = (x_{ij}) \in \Rmn$, including
vectors $x = (x_i) \in \mathbb{R}^n$, we say that $X$ is
\emph{nonnegative}, $X\ge 0$ or $X \in \Rmn_{\geq 0}$, if
all elements $x_{ij} \in \mathbb{R}_{\geq 0}$; corresponding notation are used for positive matrices.  Further, $X \in\Rnn$ is called \emph{Metzler}, if all its off-diagonal entries are nonnegative. If $X\in \Rnn$, then $\sigma(X) =
\{\lambda_1(X),\dots,\lambda_n(X)\}$ denotes its
\emph{spectrum}, where the eigenvalues are ordered by
descending absolute value, i.e., $\lambda_1(X)$ is the
eigenvalue with the largest magnitude, counting multiplicity.
In case that the magnitude of two eigenvalues coincides, we
sub-sort them by decreasing real part. A matrix $X$ is called \emph{reducible}, if there exists a permutation matrix $\pi \in \mathds{R}^{n \times n}$ such that 
$$ \pi^{T} X \pi = \begin{pmatrix}
	X_1 & \ast \\
	0 & X_2
\end{pmatrix},$$
where $X_1$ and $X_2$ are square matrices. Otherwise, $X$ is \emph{irreducible}. A spectral characterization of nonnegative matrices is given by the celebrated \emph{Perron-Frobenius Theorem} \cite{berman1979nonnegative}.
\begin{prop}[Perron-Frobenius]\label{PF}
	Let $A \in \Rnn_{\geq 0}$. Then,
	\begin{enumerate}[i.]
		\item $\lambda_1(A) \geq 0$.
		\item If $\lambda_1(A)$ has algebraic
		multiplicity $m_0$, then $A$ has $m_0$ linearly independent
		nonnegative eigenvectors related to $\lambda_1(A)$.
		\item If $A$ is irreducible, then $m_0=1$, $\lambda_1(A)>0$ has a positive eigenvectors which is the only nonnegative eigenvector to $A$. 
		\item $A$ is \emph{primitive}, i.e., there exists a $k \in \mathds{N}$ such that $A^k > 0$, if and only if $A$ is irreducible with $\lambda_1(A) > |\lambda_2(A)|$. 
	\end{enumerate}
\end{prop}
$X \in \mathds{R}^{n \times n}$ is an \emph{upper Hessenberg matrix} if $x_{ij} = 0$ for all $i,j$ with $i > j+1$. The \emph{identity matrix} of any applicable size will be denoted by $I$ and its columns, the \emph{canonical basic unit vectors} in $\mathds{R}^n$, are written as $e_1,\dots,e_n$. Using the abbreviation $(k_1:k_2) = \{k_1,\dots,k_2\} \subset \mathds{N}$, we write $X_{\mathcal{I},\mathcal{J}}$ for \emph{sub-matrices} of $X \in \Rmn$, where $\mathcal{I} \subset (1:n)$ and $\mathcal{J} \subset (1:j)$. The \emph{element-wise absolute value} is written as $|X| = (|x_{ij}|)$ and we define $$\blkdiag(X_1,X_2) := \begin{pmatrix}
	X_1 & 0 \\
	0 & X_2
\end{pmatrix}$$ for $X_1 \in \mathds{R}^{n_1 \times n_1}$ and $X_2 \in \mathds{R}^{n_2 \times n_2}$. For a subset $\mathcal{S} \subset \mathds{R}^n$, we write $\inter(\mathcal{S})$, $\partial(\mathcal{S})$, $\conv(\mathcal{S})$ and $\cone(\mathcal{S})$ for its \emph{interior}, \emph{boundary}, \emph{convex hull} and \emph{convex conic hull}. We also use $\cone(X)$, $X \in \Rmn$, to denote the convex polyhedral cone that is generated by the columns of $X$. The set of \emph{complex numbers with positive real part} is denoted by $\mathds{C}_{>0}$.

\section{Nonnegative and Metzler Hessenberg forms}
In this section, we present our main results addressing the questions that were raised in the introduction. Trivially, all $A \in \Rn$ with $\sigma(A) \subset \mathds{R}_{\geq 0}$ are similar to a nonnegative Hessenberg matrix, simply by computing its Jordan normal form. Therefore, our investigations will start with $\sigma(A) \not \subset \mathds{R}_{\geq 0}$. We begin by noticing that not every nonnegative matrix is similar to a nonnegative Hessenberg matrix. 
\begin{prop}
	\label{lem:no_hess_1}
	For $A \in \mathbb{R}^{n \times n}_{\geq 0}$, the following are equivalent:
	\begin{enumerate}[i.]
		\item $\lambda_2(A) < 0$ has geometric multiplicity $n-1$.
		\item $A = c(uv^\transp - s I)$, where $u, v \in \mathds{R}^n_{>0}$, $0\leq  s \leq \min_{i} \{u_iv_i\}$ and $c > 0$.
	\end{enumerate}
\end{prop}
\begin{proof}
	$i. \; \Rightarrow ii.$: Let $A \in \Rn$ have $\lambda_2(A) < 0$ with geometric multiplicity $n-1$. Then, $A - \lambda_2(A) I = (\lambda_1(A) - \lambda_2(A)) uv^\transp$, where $v$ and $u$ are left- and right-eigenvectors of $A$. Since $\lambda_1(A) > 0$ is the only nonnegative eigenvalue of $A$, $A$ cannot be reducible, which by  \cref{PF} implies that $u, v \in \mathds{R}^n_{>0}$. The remaining conditions ensure that $A$ is nonnegative with $A \neq 0$.\\
	$ii. \; \Rightarrow i.$: Follows by $\rk(A+csI) = 1$ and $\lambda_2(A) = -cs$.
\end{proof}
The matrices in \cref{lem:no_hess_1} have positive off-diagonal elements, which by their spectral characterization is a property that persists for all nonnegativity preserving similarity transformations. Therefore, for $n \geq 3$, these matrices provide a class of nonnegative matrices, which cannot be similar to a nonnegative Hessenberg form. This, however, raises the questions of how fundamental this property is. For the case of $n=3$, we will show that these are indeed the only such matrices. To do so, we will characterize when a second-order DT positive system is similar to a DT positive system in controller Hessenberg-form, which will result from the following general lemma that will also be used for the Metzler case later on.  

\begin{lem}\label{prop:boundary}
	Let $A \in \Rn$ be irreducible and $b \in \Rnv$ such that $Ab \neq \lambda_1(A)b$.
	Then, there exists a $T \in \Rn$ such that $T^{-1}AT = A$ and $T^{-1} b \in \partial(\Rnv)$.
\end{lem}
\begin{proof}
	It suffices to consider the case with $b > 0$ and $A$ primitive such that $\sigma(A) \subset \mathds{C}_{>0}$. If the latter is not fulfilled, one should replace $A$ with $A+ k I$ with sufficiently large $k > 0$ in the subsequent discussion. 
	
	First assume that $b \notin \inter(\cone(A))$. Then as $\lim_{s \to \infty} \cone(A+sI) = \Rn$, the continuity of $(A+sI)^{-1}b$ in $s$ implies the existence of an $s^\ast \geq 0$ such that $b \in \partial(\cone(A+s^\ast I)$. Hence, $T = A+s^\ast I \in \Rn$ with $T^{-1}AT = A$ and $T^{-1}b \in \partial(\Rnv)$. 
	
	Next, we want show that if $Ab \neq \lambda_1(A)b$ and $b \in \inter(\cone(A))$, then there exists a smallest $k \in \mathds{N}$ such $A^{-k}b \notin \inter(\cone(A))$, meaning that the state-space transformation $A^{k-1}$ yields a pair $(A,A^{-k+1}b)$ as in the previous case. To this end, assume the contrary, i.e., $A^{-k}b \geq 0$ for all $k \in \mathds{N}$. Then, by \cref{PF}, $\lambda_1(A) = \lambda_n(A^{-1}) > |\lambda_2(A)| = |\lambda_{n-1}(A^{-1})|$, which is why $0 = \lim_{k \to \infty} \left[\min_{x \in \mathcal{V}\setminus \{0\}} \left \|\frac{A^{-k}b}{\|A^{-k}b\|} - x\right \| \right]$ with $\mathcal{V}$ being the subspace spanned by the eigenvectors to $\lambda_2(A),\dots,\lambda_n(A)$. This means that there exists an $x \in \mathcal{V} \cap \Rnv \setminus \{0 \}$. However, then  $x^\ast := \lim_{k \to \infty} \frac{A^k x}{\lambda_1(A)^k} \in \mathcal{V}$ is a nonnegative eigenvector corresponding to $\lambda_1(A)$, which contradicts that $A$ is irreducible. 
\end{proof}

A direct consequence of \cref{prop:boundary} is the following Corollary. 
\begin{cor}
	\label{pos_real2}
	Let $(A,b,c)$ be a DT positive system with $b \in \mathds{R}^2_{\geq 0} \setminus \{0\}$. Then, there exists a $T \in \Rtva$ such that $(T^{-1}AT,T^{-1}b,cT)$ is a DT positive system in controller Hessenberg-form if and only if  $Ab \neq \lambda_1(A)b$ in case of $\lambda_2(A) < 0$. 
\end{cor}
\begin{proof}
	First note that if $b \not > 0$, then $A$ has to be reducible for $b$ to be an eigenvector of $A$. Hence, it suffices to consider $b > 0$ and we want to find a $p \geq 0$ such that $T^{-1}AT \geq 0$ with $T = \begin{pmatrix}
		b & p
	\end{pmatrix} \geq 0$. We begin with the case $\lambda_2(A) \geq 0$, which also includes reducible $A$. Then, $\hat{A} := A - \lambda_2(A)I \geq 0$, because otherwise there exists an $s \geq 0$ such that $0 < (\lambda_1(A)-s)(\lambda_2(A) -s) = \det(A - sI) \leq 0$. Hence, since $\rk(\hat{A}) \leq 1$,  there exists a $p \geq 0$ such that $\cone(\hat{A}) \subset \cone(T)$ with $\det(T) \neq 0$ and, thus, $T^{-1}AT \geq \lambda_2(A)I \geq 0$.
	
	This leaves us with the case, when $A$ is irreducible and $\lambda_2(A) < 0$. If $Ab \neq \lambda_1(A) b$, then \cref{prop:boundary} can be applied. Otherwise, $T^{-1}A T = \begin{pmatrix}
		\lambda_1(A) & \ast \\\
		0 & \lambda_2(A) 
	\end{pmatrix}\not \geq 0.$ 
	
\end{proof}
We are ready to completely characterize the set of matrices in $\Rtre$ that are similar to nonnegative Hessenberg matrices.

\begin{thm}
	\label{thm_non3d}
	$A \in \Rtre$ is not similar to a nonnegative Hessenberg matrix if and only if $\lambda_2(A) < 0 $ has geometric multiplicity 2. 
\end{thm}
\begin{proof}
	$\Leftarrow:$ Follows by \cref{lem:no_hess_1}.
	
	$\Rightarrow:$ First note that if $A \in \Rtre$ has an off-diagonal zero element, then there exists a permutation matrix $\pi \in \mathds{R}^{3 \times 3}$ such that $\pi^\transp A \pi$ is a nonnegative Hessenberg matrix. Therefore, we can assume that $$A = \begin{pmatrix}
		A_{11} & v_r\\
		v_l^{T} & a_{33}
	\end{pmatrix} = \begin{pmatrix}
		a_{11} & w_l^{T}\\
		w_r & A_{22}
	\end{pmatrix}$$ with $a_{ij} >0, \ 1\leq i \neq j \leq n.$ Next, if there does not exists a $\tilde{T} \in \Rtva$ such that $\blkdiag(\tilde{T},1)^{-1} \ A \  \blkdiag(\tilde{T},1)$ or $\blkdiag(1,\tilde{T})^{-1} \ A \ \blkdiag(1,\tilde{T})$ is nonnegative with at least one off-diagonal zero, then by \cref{pos_real2}, $v_r$ and $v_l^\transp$, as well as $w_r$ and $w_l^\transp$ are right- and left eigenvectors to $A_{11}$ and $A_{22}$, respectively, and $\lambda_2(A_{11}), \lambda_2(A_{22}) < 0$. 
	By computing Jordan normal forms of $A_{11}$ and $A_{22}$, we can assume that
	\begin{equation}
		\label{eq_jord}
		T_1^{-1}AT_1 = 
		\begin{pmatrix}
			\lambda_2(A_{11}) & 0 & 0 \\
			\geq 0     & \lambda_1(A_{11}) & 1 \\
			0     &  >0        & a_{33}
		\end{pmatrix}, \; \ T_2^{-1}AT_2 = \begin{pmatrix}
			a_{11} & > 0 & 0 \\
			1    & \lambda_1(A_{22})& \geq 0 \\
			0     &  0        & \lambda_2(A_{22}) 
		\end{pmatrix}
	\end{equation}
	with $T_1 = \blkdiag(T_{11},1)$ and $T_2 = \blkdiag(1,T_{22})$, and thus $\lambda_2(A_{11}),\lambda_2(A_{22}) \in \sigma(A)$. We want to show now through proof by contradiction that $\lambda_2(A) = \lambda_2(A_{11}) = \lambda_2(A_{22})$ with geometric multiplicity $2$. Assume that $\lambda_2(A_{11}) > \lambda_2(A_{22})$. Then, $\rk(A_{22} - \lambda_2(A_{11}) I) = 2$ and $\rk(A_{11} - \lambda_2(A_{11}) I) = 1$. Hence, if $(A - \lambda_2(A_{11}) I)x = 0$,  with $x \neq 0$, then $$\begin{pmatrix}
		x_2\\
		x_3
	\end{pmatrix} = -x_1(A_{22} - \lambda_2(A_{11}) I)^{-1}w_r =\dfrac{x_1}{\lambda_2(A_{11})-\lambda_2(A_{22})}w_r.$$
	Thus, choosing $x_1 >0$ implies $x \in \mathbb{R}^3_{> 0}$, which requires that $\lambda_2(A_{11}) \geq 0$ and, therefore,  $\lambda_2(A_{11}) \leq \lambda_2(A_{22})$. Analogously, $\lambda_2(A_{22}) \leq \lambda_2(A_{11})$, which proves that $\lambda_2(A_{22}) = \lambda_2(A_{11})$. Finally, let us look into the eigenspace of $\lambda_2(A_{22})$. We partition $T_{11} = \begin{pmatrix}
		v_{11} & v_{12}
	\end{pmatrix}$ and $T_{22} = \begin{pmatrix}
		v_{21} & v_{22}
	\end{pmatrix}$, which by \cref{eq_jord} implies that $\begin{pmatrix}
		v_{11}\\
		0
	\end{pmatrix}
	\ \text{and} \ \begin{pmatrix}
		0  \\
		v_{22}
	\end{pmatrix}$ are eigenvectors of $A$ corresponding to the eigenvalue $\lambda_2(A_{11})$. If these eigenvectors were not linearly independent, we could choose them such that $$\begin{pmatrix}
		v_{11}\\
		0
	\end{pmatrix} = \begin{pmatrix}
		0  \\
		v_{22}
	\end{pmatrix} = \begin{pmatrix}
		0 \\
		\star\\
		0
	\end{pmatrix} \geq 0, $$ which would require that $\lambda_2(A_{11}) \geq 0$ and, therefore, provides a  contradiction.
\end{proof}
Since Metzler matrices allow to freely shift the spectrum, we get the following analogue. 
\begin{cor}
	For every Metzler $A \in \mathds{R}^{3 \times 3}$ there exists a $T \in \Rtre$ such that $T^{-1}AT$ is a Metzler Hessenberg matrix. 
\end{cor}
In order to show that the same holds true in $\mathds{R}^{4 \times 4}$, we require an extension of \cref{pos_real2} to third-order CT positive systems. We will prove this via the following general lemma. 

\begin{lem}
	Let $A \in \Rn$ be irreducible with $\sigma(A) \subset \mathds{R}_{\geq 0}$ and $b \in \mathds{R}^n_{>0}$ be such that $Ab = \lambda_1(A)b$. Then, there exists a $T \in \Rn$ such that $T^{-1}AT \geq 0$ and $T^{-1}b = e_1$. \label{lem:irred_b_EV}
\end{lem}
\begin{proof}
	Let $V^{-1}AV = J$ be a Jordan normal form with ones above the diagonal and decreasing diagonal entries, where $V = \begin{pmatrix}
		b & v_2 & \dots &v_n
	\end{pmatrix}$. Then, there exists an $\alpha \in \mathds{R}^n_{\geq 0}$ with $\alpha_1 = 0$ such that 
	\begin{align*}
		T = \begin{pmatrix}
			b & \alpha_2 b + v_2 & \dots & \alpha_n b + v_n
		\end{pmatrix} = V (I+e_1 \alpha^\transp)\geq 0.
	\end{align*}
	Hence, $T^{-1}AT = (I - e_1\alpha^\transp)J(I+e_1\alpha^\transp) = J + e_1 \beta^\transp$ with $\beta_i = J^\transp \alpha + \lambda_1(A) \alpha$, i.e., $\beta_1 = 0$ and $\beta_i = \alpha_i(\lambda_1(A)-\lambda_i(A))$ or $\beta_i = (\lambda_1(A)-\lambda_i(A)) \alpha_i-\alpha_{i-1}$ for $i \geq 2$. Since $\lambda_1(A) - \lambda_i(A) > 0$ for all $i\geq 2$ by \cref{PF}, we can choose a sufficiently large $\alpha_i \geq \frac{\alpha_{i-1}}{\lambda_1(A)-\lambda_i(A))} > 0$ for the claim to hold. 
\end{proof}

\begin{thm}
	\label{thm:Hess_pos_sys3d}
	Let $(A,b,c)$ be a third-order CT positive system. Then, there exists a $T \in \Rtre$ such that $(T^{-1}AT,T^{-1}b,cT)$ is a CT positive system in controller-Hessenberg form if and only if $Ab \neq \lambda_1(A)b$ whenever $\sigma(A) \not \subset \mathds{R}$. 
\end{thm}
\begin{proof}
	It suffices to assume that $A \in \Rtre$ and $\sigma(A) \subset \mathds{C}_{>0}$. We want to find a $T = \begin{pmatrix}
		b & p & q\end{pmatrix} \in \Rtre$ such that $T^{-1} A T = \begin{pmatrix}\ast & \tilde{c}\\
		\tilde{b} & \tilde{A}
	\end{pmatrix}\geq 0$. Then applying \cref{pos_real2} to $(\tilde{A},\tilde{b},\tilde{c})$ yields the result. Obviously, such a $T$ cannot exist if $Ab = \lambda_1(A)b$ and $\sigma(A) \not \subset \mathds{R}$, which is why we want to consider all other cases. 
	
	Let us first assume that $A$ is reducible. As in the proof to \cref{pos_real2},  $\hat{A} := A - \lambda_3(A)I \geq 0$ with $\rk(\hat{A}) \leq 2$. Therefore, there exist $p,q  \geq 0$ such that $\cone(\hat{A}) \subset \cone(\{p,q\}) \subset \cone(T)$ and $\det(T) \neq 0$, which shows that $T^{-1}AT \geq \lambda_3(A)I \geq 0$. 
	
	Let $A$ be irreducible, now. If $Ab = \lambda_1(A)b$ and $\sigma(A) \subset \mathds{R}_{\geq 0}$, then the result follows by \cref{lem:irred_b_EV}. Otherwise, by our assumption and \cref{prop:boundary}, we can assume that $b_3 = 0$, $b_1,b_2 > 0$ and that the diagonal entries of $A$ are sufficiently large such that
	$${A}_b := {A} - b \alpha^\transp = \begin{pmatrix}
		x_1 & 0 & z_1 \\
		0   & y_2 & z_2 \\
		x_3 & y_3 & z_3
	\end{pmatrix} =: \begin{pmatrix}
		x & y & z
	\end{pmatrix} \Rtre,$$
	with $\alpha \in \Rtrev$ and $z_1 = 0$ or $z_2 = 0$. Hence, $A_b$ is reducible and, thus, as above there exist $p,q \geq 0$ such that $T^{-1}A_b T  \geq 0$, which implies that $T^{-1}AT = T^{-1}A_bT + e_1 \alpha^\transp T \geq 0$.

\end{proof}
\cref{thm:Hess_pos_sys3d} is a new characterization of a standard-form for third-order CT positive system. In particular, it implies that every controllable third-order CT positive system has a controller Hessenberg-form, which keeps the system realization positive: this has only be known for minimal second-order systems \cite{farina2011positive}. As in the proof to \cref{thm_non3d}, we will use this result to prove our final characterization.

\begin{thm}
	\label{thm_metz4d}
	For every Metzler $A \in \mathbb{R}^{4 \times 4}$ there exists a $T \in \Rfyra$ such that $T^{-1}AT$ is a Metzler Hessenberg matrix. 
\end{thm}
\begin{proof}
	It suffices to show that there exists a permutation matrix $\pi \in \Rfyra$ such that \cref{thm:Hess_pos_sys3d} can be applied to $(A_{\pi},b_{\pi},c_{\pi})$, where 
	$\pi^\transp  A \pi = \begin{pmatrix}
		A_{\pi} & b_{\pi}\\
		c_{\pi} & d_{\pi}
	\end{pmatrix}$.
	Let us assume the contrary. Then, $A$ must be irreducible and for the partition 
	\begin{align*}
		A = \begin{pmatrix}
			A_{11} & v_r\\
			v_l^\transp & a_{44}\end{pmatrix} = \begin{pmatrix}
			a_{11} & w_l^\transp\\
			w_r & A_{22}
		\end{pmatrix},
	\end{align*}
	it must hold that $v_r$ and $v_l^\transp$, as well as $w_r$ and $w_l^\transp$ are positive right- and left eigenvectors to irreducible $A_{11}$ and $A_{22}$, respectively with $\sigma(A_{11}),\sigma(A_{22})  \not \subset \mathds{R}$. Hence, by computing real diagonal forms to $A_{11}$ and $A_{22}$, we can assume that
	\begin{align*}
		T_1^{-1}AT_1 = \begin{pmatrix}
			\alpha & -\beta & 0 & 0\\
			\beta & \alpha  &  0  & 0\\
			0  & 0      &   \lambda_1(A_{11}) & 1\\
			0 & 0 & >0 & a_{44}
		\end{pmatrix}, \; T_2^{-1}AT_2 = \begin{pmatrix}
			a_{11} & >0 & 0 & 0\\
			1        & \lambda_1(A_{22}) & 0 & 0 \\
			0        &  0                           &  \alpha & -\beta\\
			0        &   0 &                     \beta & \alpha
		\end{pmatrix}
	\end{align*}
	for some $\alpha, \beta \in \mathds{R}$, $T_1 = \blkdiag(T_{11},1)$, $T_2 = \blkdiag(1,T_{22})$, where ${T_{11}}_{(1:3),\{3\}} = v_r$ and ${T_{22}}_{(1:3),\{1\}} = w_r$. Further, diagonalizing the remaining Metzler blocks $(T_1^{-1}AT_1)_{(3:4),(3:4)}$ and  $(T_1^{-1}AT_1)_{(1:2),(1:2)}$ with $\hat{T}_{11} = \begin{pmatrix}p & \ast \end{pmatrix} \in \mathds{R}^{2 \times 2}$ and $\hat{T}_{22} = \begin{pmatrix}\ast & q \end{pmatrix}\in \mathds{R}^{2 \times 2}$, $p, q  \in \mathds{R}^{2}_{> 0}$, respectively, then shows that 
	\begin{align*}
		T_1 \blkdiag(I,\hat{T}_{11}) e_3 &= \begin{pmatrix}
			{v_r}_1 p_1  & {v_r}_2 p_1 & {v_r}_3 p_1 & p_2
		\end{pmatrix}^\transp\\
		\intertext{and}
		T_2 \blkdiag(\hat{T}_{22},I) e_2 &= \begin{pmatrix}
			q_1 & {w_r}_1 q_2  & {w_r}_2 q_2 & {w_r}_3 q_2	\end{pmatrix}^\transp
	\end{align*}
	are positive eigenvectors to $\lambda_1(A)$, which by \cref{PF} and the irreducibility of $A$ are linearly dependent. Consequently, $A_{11}e_1 = \begin{pmatrix}
		a_{11} & {v_{r}}_2 k & {v_r}_3 k
	\end{pmatrix}^\transp$ for some $k > 0$, which is why there exists an $s \geq 0$ such that $\hat{A}_{11} := A_{11} - k v_r e_1^\transp + sI \geq 0$ is reducible. Thus, by \cref{thm:Hess_pos_sys3d}, there exists a $T_p \in \Rtre$ such that $(T_p^{-1}\hat{A}_{11}T_p,T_p^{-1}v_r,v_lT_p)$ is a CT positive system in controller-Hessenberg form. However, as this also implies that $T_p^{-1}A_{11}T_p$ is a Metzler Hessenberg matrix, we have a contradiction. 
\end{proof}

Note that while \cref{thm:Hess_pos_sys3d} remains true in DT if $A$ is sufficiently diagonally dominant, in general this is not true, even if $A$ is similar to a nonnegative Hessenberg matrix. To see this, let us consider
\begin{align}
	A = \begin{pmatrix}
		0  &   0  &  14\\ 
		0  &    6  &   0\\
		15 &    4 &     6
	\end{pmatrix} = \begin{pmatrix}
		x & y & z
	\end{pmatrix}, \ b = \begin{pmatrix}
		1\\
		1\\
		0
	\end{pmatrix}. \label{eq:counter_ex}
\end{align}
The eigenvalues of $A$ are distinct, which by \cref{thm_non3d} implies that $A$ is similar to a nonnegative Hessenberg matrix. If there exists a $T = \begin{pmatrix}
	b & p & q
\end{pmatrix} \in \Rtre$ with $T^{-1}AT \in \Rtre$, it must hold that $A^kb \in \cone(T)$ for all $k \in \mathds{N}$. Equivalently, $\Pi(A^kb) \in \conv(\{\Pi(b), \Pi(q),\Pi(q)\})$, $k \in \mathds{N}$, with $\Pi(x) := \frac{x}{\sum_{i=1}^3 x_i}$ denoting the projection onto the simplex simplex $\mathcal{S} = \{x \in \Rtrev: \sum_{i=1}^3 x_i = 1 \}$. Since $\mathcal{S}$ can be mapped bijectively onto $\conv(\{e_1,e_2\})$, this is possible if and only if the same applies when replacing $\Pi$ with $\hat{\Pi}(x) := \Pi(x)_{(2:3)}$. In other words, for $T$ to exist, there has to be a triangle within $\conv(\{e_1,e_2\})$ that has $\hat{\Pi}(b)$ as a corner point and surrounds all $\hat{\Pi}(A^kb)$. That this is impossible can be seen in \cref{fig_count}. In particular, since $c^\transp \in \Rtrev$ can be chosen arbitrarily, even removing the nonnegativity constraint on $T$ will not always allow us to achieve a DT positive system $(T^{-1}AT,T^{-1}b,cT)$ in controller-Hessenberg form.

\begin{figure}
	\centering
	\begin{tikzpicture}
		\begin{axis}[xmin =0, xmax =1, ymin =0, ymax = 1,width = 12cm, height = 8 cm]
			\addplot+[only marks, thin, blue, mark = *, mark options={fill=blue}] coordinates {   
				(5.0000000e-01, 0.0000000e+00)
				(2.4000000e-01,   7.6000000e-01)
				(8.1818182e-02,   3.1363636e-01)
				(3.0379747e-02,   6.9789030e-01)
				(9.9401749e-03,   4.5724804e-01)
				(3.4601522e-03,   6.2515018e-01)
				(1.1464765e-03,   5.1553759e-01)
				(3.9146805e-04,   5.8886733e-01)
				(1.3090841e-04,   5.4039031e-01)
				(4.4372480e-05,   5.7255958e-01)
				(0.0000000e-00,   5.5972972e-01)}; \label{marker:Ab}  
			\addplot [dashed, black] coordinates { (0,0) (1,0) (0,1) (0,0)}; \label{line:triangle}
			\node[inner sep=0cm] at (axis cs: 0.5,0) (b){};
			\node at (b) [above,xshift = .3 cm]{$\hat{\Pi}(b)$};
			\node[inner sep=0cm] at (axis cs: 0.24,0.76) (Ab){};
			\node[inner sep=0cm] at (axis cs: 0,0.5597) (Ainfb){};
			\node at (Ab)[right, xshift = .1 cm]{$\hat{\Pi}(Ab)$};
			\node[inner sep=0cm] at (axis cs: 8.1818182e-02,3.1363636e-01) (A2b){};
			\node at (A2b)[below]{$\hat{\Pi}(A^2b)$};
			\node at (Ainfb)[right]{$\hat{\Pi}(A^\infty b)$};
			\node[inner sep=0 cm] at (axis cs: 0,0.375) (d){};
			\addplot[dotted,blue,thin] coordinates {(0.5,0)  (0.24,0.76) (0,0.5597) (0.5,0)}; \label{line:triangel_narrow}
		\end{axis}
		
	\end{tikzpicture}
	\caption{Illustration to counter example \cref{eq:counter_ex}: If there exists a $T = (b \; p \; q) \geq 0$ with $T^{-1}AT \geq 0$, then $\ref{marker:Ab}\; \hat{\Pi}(A^k b) \in \conv(\{  \hat{\Pi}(b), \hat{\Pi}(p), \hat{\Pi}(q)\}) \subset \conv(\{e_1,e_2\})$ for all $k \in \mathds{N}$. However, as $\hat{\Pi}(b)$, $\hat{\Pi}(Ab), \lim_{k \to \infty} \hat{\Pi}(A^k b) \in \partial(\conv(\{e_1,e_2\}))$, this is impossible as indicated by the triangle \ref{line:triangel_narrow}.  \label{fig_count}}
\end{figure}
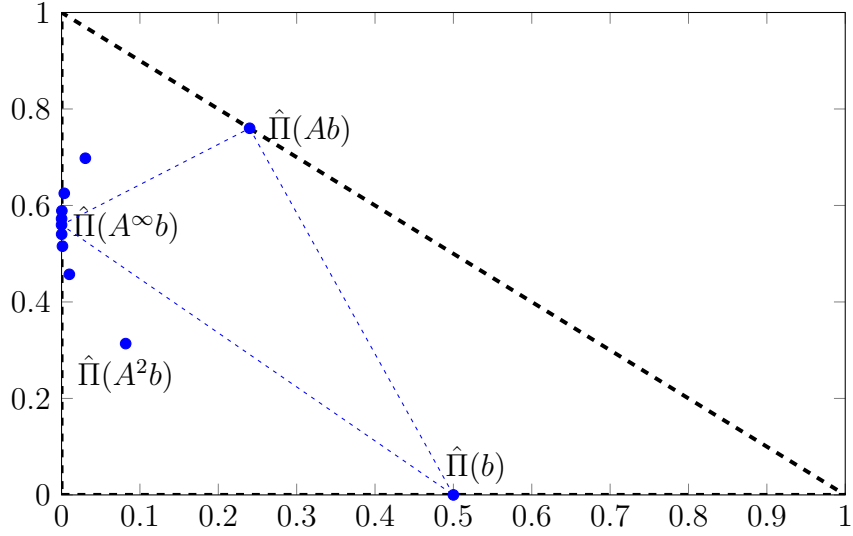

Unfortunately, we do not know whether \cref{thm:Hess_pos_sys3d} extends beyond systems with state-space dimension $3$ and as such it is hard to extend \cref{thm_metz4d} to $n > 4$. To elaborate on this issue, note that \cref{thm:Hess_pos_sys3d} is based on proving that for all pairs of irreducible $A \in  \mathds{R}^{3 \times 3}_{\geq 0}$ and $b \in \partial(\mathds{R}^{3}_{\geq 0})$ as well as all pairs of reducible $A \in  \mathds{R}^{3 \times 3}_{\geq 0}$ and $b \in \mathds{R}^{3}_{\geq 0}$, there exist $T = \begin{pmatrix}
	b & t_2 & t_3
\end{pmatrix}$, $N \in \mathds{R}^{3 \times 3}_{\geq 0}$ and $s \in \mathds{R}$ such that 
\begin{equation}
	A - sI = TN. \label{eq:prob_TN}
\end{equation}
We were able to establish this result, because rank and nonnegative rank of $A-sI$ coincide if it they are bounded by $2$. Extending such a characterization to higher dimensions is a crucial challenges for further advances and lies within the scope of nonnegative matrix factorization \cite{gillis2020nonnegative}. Nonetheless, it is worth pointing out the following variants.
\begin{prop} \label{prop:boundary_e1}
	Let $b \in \mathds{R}^n$ and $A \in \mathds{R}^{n \times n}$ be diagonalizable with $V^{-1}AV = D$ such that $|V^{-1}e_1| > 0$ and $|V^{-1}b| > 0$. Then, there exists a $T \in \mathds{R}^{n \times n}$ such that $T^{-1}AT = A$ and $T^{-1}b = e_1$.
\end{prop}
\begin{proof}
	Let $\alpha := V^{-1}e_1$ and $\beta := V^{-1} b$. Then, $T := V E V^{-1}$ with $E : = \diag(\frac{\beta_1}{\alpha_1},\dots,\frac{\beta_n}{\alpha_n})$ is invertible and $T^{-1}b = e_1$. Further, since $E^{-1} D E = D$, it follows that $T^{-1}AT = VDV^{-1} = A$.
\end{proof}
In conjunction with \cref{thm_metz4d}, \cref{prop:boundary_e1} yields the following partial extension of \cref{thm:Hess_pos_sys3d}.
\begin{cor}
	Let $b \in \mathds{R}^{4}_{\geq 0}$ and Metzler $A \in \mathds{R}^{4 \times 4}$ be as in  \cref{prop:boundary_e1}. Then, there exists a $T \in \mathds{R}^{4 \times 4}$ such that $(T^{-1}AT,T^{-1}b)$ is in CT positive controller-Hessenberg form. 
\end{cor}

\section{Concluding remarks}
We have characterized the similarity of nonnegative and Metzler matrices to nonnegative and Metzler Hessenberg matrices of up to dimension $4$, which can be used towards new standard forms for positive systems. As many of the discussed techniques and tools extend beyond these dimensions, such forms can also practically be found for matrices of larger dimensions. In particular, numerical experiments indicate that imposing our derived block-diagonal form for the similarity transformation significantly helps an alternating projection approach to converge to a feasible invertible transformation. We, further, presented a new related research problem within the scope of nonnegative matrix factorization, where future insights may lead to extensions of our results. Finally, it would be interesting to see how our results can be exploited in the nonnegative/Metzler inverse eigenvalue problem.

\section*{Acknowledgements}
This work was supported by the ELLIIT Excellence Center and by the Swedish Research Council through the LCCC Linnaeus Center. It was also supported by the European Research Council under the ERC Advanced Grant ScalableControl n.834142 and by SSF under the grant RIT15-0091 SoPhy.

\bibliographystyle{plain}

\end{document}